\documentclass{amsart}
\usepackage{amsfonts,amssymb,amsmath,amsthm}
\usepackage{url}
\usepackage{enumerate}
\usepackage{hyperref}
\usepackage{stackengine}
\usepackage{scalerel}
\newtheorem{theorem}{Theorem}[section]
\newtheorem{lemma}[theorem]{Lemma}

\newtheorem{corollary}[theorem]{Corollary}
\theoremstyle{definition}
\newtheorem{definition}[theorem]{Definition}
\newtheorem{remark}[theorem]{Remark}

\numberwithin{equation}{section}

\begin{document}
\title[Exploring certain geometric and harmonic properties\ldots]{Exploring certain geometric and harmonic properties of the Berger-type metric conformal deformation on the Para-K\"{a}hler-Norden manifold}
\author{Abderrahim Zagane}
\address{Laboratory Analysis, geometry and its applications, Department of Mathematics, Ahmed Zabana Relizane University, 48000, Relizane-Algeria}
\email{Zaganeabr2018@gmail.com}
\author{Fethi Latti}
\address{Salhi Ahmed Naama University Center, Department of mathematics, 45000, Naama-Algeria}
\email{etafati@hotmail.fr}

\begin{abstract}
This work presents a novel class of metrics on a para-K\"{a}hler-Norden manifold $(M^{2m},F,g)$, derived from a conformal deformation of the Berger-type metric associated with the metric $g$. Initially, we examine the Levi-Civita link associated with this metric. Secondly, we delineate all varieties of curvature for a manifold $M$ equipped with a conformal deformation of  Berger-type metric for $g$.  Finally, we studied a certain class of harmonic maps.
	
	
\end{abstract}

\maketitle 

\section{Introduction}

Consider a $n$-dimensional Riemannian manifold $(M^{n}, g)$ and $\nabla$ represent the Levi-Civita connection associated with $g$, as defined by the Koszul formula.

Let $\alpha$ be a smooth function on $M$, the gradient of $\alpha$, denoted by $grad^{}\alpha$ is defined by
\begin{equation*}
g(grad^{}\alpha,X)=X(\alpha).
\end{equation*}
If $\{e_{i}\}_{i=1,\ldots,n}$ is a local orthonormal frame on $(M^{n},g)$, then 
the $grad^{}\alpha$ is expressed by 
 \begin{equation*}
grad^{}\alpha= \sum_{i=1}^{n} e_{i}(\alpha)e_{i}.
 \end{equation*}
The Hessian of $\alpha$, denoted by $Hess_{\alpha}$ is given by 
\begin{equation*} \label{Hess}
Hess_{\alpha}(X,Y)=g(\nabla _{X}grad^{}\alpha,Y),
\end{equation*}
for each  vector fields $X$ and $Y$ on $M$.

The Laplacian of $\alpha$ on $M$, denoted $\Delta (\alpha)$, is expressed by
\begin{equation*}
\Delta (\alpha)=Tr_{g}(Hess_{\alpha})=\sum_{i=1}^{n} g(\nabla_{e_{i}}grad^{}\alpha,e_{i}).
\end{equation*}

A vector field $X$ on $M$ is called a killing vector field if $ L_{X}g = 0$, where $L_{X}$ denotes the Lie differentiation with respect to $X$ or equivalently
\begin{equation*}
g(\nabla _{Y}X,Z)+g(\nabla _{Z}X,Y)=0,
\end{equation*}
for each vector fields $Y$ and $Z$ on $M$. 

If $grad^{}\alpha$ is a Killing vector field, then the smooth function $\alpha$ on $M$ is called a Killing potential. In this case, the $grad^{}\alpha$ is also called a killing gradient sse \cite{Cra}. Furthermore, $\alpha$ is Killing potential if and only if $Hess_{\alpha}=0$ or, equivalently, $\nabla _{X}grad^{}\alpha=0$, for each vector fields $X$ on $M$.

Given a $2m$-dimensional differentiable manifold $M^{2m}$. An almost para-complex structure $F$ is a $(1,1)$ tensor field on $M$ such that $F^{2}=I$ ($I$ is the identity tensor field of type $(1,1)$ on $M$) and the two eigenbundles $TM^{+}$, $TM^{-}$ associated to the two eigenvalues $+1$, $-1$ of $F$, respectively have the same rank. The pair $(M,F)$ is called an almost para-complex manifold.

An almost para-complex Norden manifold $(M^{2m}, F, g)$ is an almost para-complex manifold $(M^{2m},F)$ with a Riemannian metric $g$ such that:
\begin{eqnarray*}
g(F X, Y)= g(X, F Y),
\end{eqnarray*}
for each vector fields $X$ and $Y$ on $M$, then the Riemannian metric $g$ is called para-Norden metricin or pure metric with respect to the almost paracomplex structure $F$. 

A para-K\"{a}hler-Norden manifold is an almost para-complex Norden manifold $(M^{2m}, F, g)$ such that $\nabla F = 0$, where $\nabla$ is the Levi-Civita connection of $g$ \cite{S.I.E, S.G.I}.

In para-K\"{a}hler-Norden manifold, the Riemannian curvature tensor is pure \cite{S.I.E}. Furthermore, we have
\begin{eqnarray}\label{curva-pure}
\begin {aligned}
&R(F X,Y)Z = R(X,F Y)Z =R(X,Y)F Z= F R(X,Y)Z,
\end{aligned} 
\end{eqnarray}		
for each vector fields $X$, $Y$ and $Z$ on $M$.

In previous work \cite{D.Z4}, we introduced a new class of metrics on the Riemannian manifold $(M^{m},g)$, which are defined as follows:  
\begin{equation*}
G(X,Y) =fg(X,Y)+g(\xi,X)g(\xi,Y),
\end{equation*}
for each vector fields $X, Y$ and $\xi$ on $M$, such that $\xi$ is a unit parallel vector field on $M$ and $\xi(f)=0$, with $f$ being a strictly positive smooth function on $M$. Several geometric properties related to this metric were studied.

Furthermore, in an independent study \cite{Zag26}, we introduced an innovative class of metric on the $B$-manifold $(M^{2m},\varphi,g)$, expressed by: 
\begin{equation*}
{}^{GB}\!g(X,Y) =g(X,Y)+fg(X, \varphi\xi)g(Y, \varphi\xi),
\end{equation*}
for each vector fields $X, Y$ and $\xi$ on $M$, such that $\xi$ is a unit parallel vector field on $M$ and $(\varphi\xi)(f)=0$, where $f$ is a positive smooth function on $M$. We investigated certain geometric and harmonic characteristics of the $B$-manifold concerning this metric. 

Drawing from the study as mentioned earlier, we present a novel class of metric on the para-K\"{a}hler-Norden manifold $(M^{2m},F,g)$, which is represented as follows: 
\begin{equation*}
g^{\alpha}(X,Y)=\alpha\big(g(X,Y)+g(X,FV)g(Y,FV)\big),
\end{equation*}
for each vector fields $X,Y$ on $M$ and $V$ is a unit parallel vector field on $M$ such that 
$(FU)(\alpha)=0$. We call $g^{\alpha}$ a conformal deformation of Berger-type metric of $g$.

In this study, we explore a Berger-type metric conformal deformation of $g$ on the para-K\"{a}hler-Norden manifold $(M^{2m},F,g)$. After describing the introduction, we present the fundamental properties of the conformal deformation of the Berger-type metric of $g$ (including the Levi-Civita connection) in section \ref{Sec2}. In Section \ref{Sec3}, we study all types of curvature. First, we examine the curvature tensor and the sectional curvature. Furthermore, we characterize the Ricci tensor, the Ricci curvature, and the scalar curvature. In Section \ref{Sec4}, we explore the harmonicity of the identity map between two manifolds, one of which is a Riemannian manifold and the other a para-K\"{a}hler-Norden manifold, or vice versa. We also talk about the harmonicity of a smooth map between two Riemannian manifolds, one of which has a Berger-type conformal deformation. In the last part of this section, we study the harmonicity of the curve in the Riemannian manifold endowed with a Berger-type metric conformal deformation of $g$. In the last section, we study the biharmonicity of the identity map, where we establish the necessary and sufficient conditions under which the identity map is a proper biharmonic.

\section{Fundamental properties of the Berger-type metric conformal deformation}\label{Sec2}
\begin{definition}
Given an almost para-complex Norden manifold $(M^{2m},F,g)$ and a strictly positive smooth function $\alpha : M\rightarrow]0,+\infty[$.  A Berger-type metric conformal deformation of $g$ on $M$ noted $g^{\alpha}$ is define by
\begin{eqnarray*}
g^{\alpha}(X,Y)=\alpha(x)\left( g(X,Y)+g(X,FV)g(Y,FV)\right) ,
\end{eqnarray*}
for each vector fields $X,Y$ on $M$ and $V$ is a unit parallel vector field on $M$ such that 
\begin{eqnarray}\label{eq_A}
\begin {aligned}
&(FV)(\alpha)=0.
\end{aligned}
\end{eqnarray}
\end{definition}
We conclude from the definition the following:
\begin{eqnarray}\label{eq_B}
\left\{\begin {aligned}
&g(FV,FV)= 1, \\
&g(X,FV)= \frac{1}{2\alpha}g^{\alpha}(X,FV),\\
&X(\alpha)=\frac{1}{\alpha}g^{\alpha}(X,grad^{}\alpha),\\
&Hess_{\alpha}(X,FV)=0, \\
&R(X,Y)V=0,
\end{aligned} \right.
\end{eqnarray}
for each vector field $X$ on $M$.

\begin{lemma}\label{lem_1}
Given a para-K\"{a}hler-Norden manifold $(M^{2m},F,g)$, then we have
\begin{eqnarray}\label{eq_C}
X g^{\alpha}(Y,Z) &=&g^{\alpha}(\nabla_{X}Y,Z)+g^{\alpha}(Y,\nabla_{X}Z)+\frac{X(\alpha)}{\alpha}g^{\alpha}(Y,Z),
\end{eqnarray}
for each vector fields $X,Y$ and $Z$ on $M$.
\end{lemma}

We will now calculate the Levi-Civita connection $\widetilde{\nabla}$ for $(M^{2m},g^{\alpha})$ which is given by the following theorem.

\begin{theorem}\label{th_1}
Given a para-K\"{a}hler-Norden manifold $(M^{2m},F,g)$, the Levi-Civita connection $\widetilde{\nabla}$ of $(M^{2m}, g^{\alpha})$, is expressed by:
\begin{eqnarray}\label{eq_D}
\widetilde{\nabla}_{X}Y &=& \nabla_{X}Y+\frac{X(\alpha)}{2\alpha}Y+\frac{Y(\alpha)}{2\alpha}X-\frac{1}{2\alpha^{2}}g^{\alpha}(X,Y)grad^{}\alpha,
\end{eqnarray}
for each vector fields $X$ and $Y$ on $M$, where $\nabla$ is the Levi-Civita connection of $g$.
\end{theorem}

\begin{proof}
By \eqref{eq_C},  and using Kozul formula we obtain:	
\begin{eqnarray*}
2g^{\alpha}(\widetilde{\nabla}_{X}Y,Z) &=&X g^{\alpha}(Y,Z)+Y g^{\alpha}(Z,X)-Z g^{\alpha}(X,Y)+ g^{\alpha}(Z,[X,Y])  \\
&&+ g^{\alpha}(Y,[Z,X])- g^{\alpha}(X,[Z,Y])  \\
&=& g^{\alpha}(\nabla_{X}Y,Z)+g^{\alpha}(\nabla_{X}Z,Y)+\frac{X(\alpha)}{\alpha}g^{\alpha}(Y,Z) \\
&& +g^{\alpha}(\nabla_{Y}Z,X)+g^{\alpha}(\nabla_{Y}X,Z)+\frac{Y(\alpha)}{\alpha}g^{\alpha}(Z,X) \\
&& -g^{\alpha}(\nabla_{Z}X,Y)-g^{\alpha}(\nabla_{Z}Y,X)-\frac{Z(\alpha)}{\alpha}g^{\alpha}(X,Y) \\
&& +g^{\alpha}(Z,\nabla_{X}Y)-g^{\alpha}(Z,\nabla_{Y}X)+g^{\alpha}(Y,\nabla_{Z}X) \\
&& -g^{\alpha}(Y,\nabla_{X}Z)-g^{\alpha}(X,\nabla_{Y}Z)+g^{\alpha}(X,\nabla_{Z}Y).
\end{eqnarray*}
From, \eqref{eq_B}, we get,
\begin{eqnarray*}
2g^{\alpha}(\widetilde{\nabla}_{X}Y,Z) &=& 2g^{\alpha}(\nabla_{X}Y,Z)+\frac{X(\alpha)}{\alpha}g^{\alpha}(Y,Z)+\frac{Y(\alpha)}{\alpha}g^{\alpha}(X,Z) \\
&&-\frac{1}{\alpha^{2}}g^{\alpha}(X,Y)g^{\alpha}(grad^{}\alpha,Z).
\end{eqnarray*}
This completes the proof.
\end{proof}

Using \eqref{eq_B} and \eqref{eq_D}, we obtain the following: 
\begin{corollary}\label{cor_01}
Given a para-K\"{a}hler-Norden manifold $(M^{2m},F,g)$, then the following equality holds.
\begin{eqnarray}\label{eq_E}
\widetilde{\nabla}_{X}grad^{}\alpha &=&\nabla_{X}grad^{}\alpha+\frac{\|grad^{}\alpha\|^{2}}{2\alpha}X, 
\end{eqnarray}
for each vector field $X$ on $M$.
\end{corollary}

\section{Curvatures of Berger-type metric conformal deformation}\label{Sec3}
Below we will calculate all the forms of the curvature tensor for $(M^{2m}, g^{\alpha})$.
\begin{theorem}\label{th_2}
Given a para-K\"{a}hler-Norden manifold $(M^{2m},F,g)$, then the corresponding Riemannian curvature tensor  $\widetilde{R}$, is expressed by:
\begin{eqnarray}\label{eq_F}
\widetilde{R}(X,Y)Z&=&R(X,Y)Z-\frac{1}{2\alpha^{2}}g^{\alpha}(Y,Z)\nabla_{X}grad^{}\alpha+\frac{1}{2\alpha^{2}}g^{\alpha}(X,Z)\nabla_{Y} grad^{}\alpha\nonumber\\
&&+\left( \frac{3Y(\alpha)Z(\alpha)}{4\alpha^{2}}-\frac{1}{2\alpha}Hess_{\alpha}(Y,Z)-\frac{\|grad^{}\alpha\|^{2}}{4\alpha^{3}}g^{\alpha}(Y,Z)\right) X\nonumber\\
&&-\left( \frac{3X(\alpha)Z(\alpha)}{4\alpha^{2}}-\frac{1}{2\alpha}Hess_{\alpha}(X,Z)-\frac{\|grad^{}\alpha\|^{2}}{4\alpha^{3}}g^{\alpha}(X,Z)\right) Y\nonumber\\
&&+\left(\frac{3X(\alpha)}{4\alpha^{3}}g^{\alpha}(Y,Z)-\frac{3Y(\alpha)}{4\alpha^{3}}g^{\alpha}(X,Z)\right)grad^{}\alpha.
\end{eqnarray}
for each vector fields $X,Y$ and $Z$ on $M$, where $R$ denote the curvature tensor of $(M^{2m},F,g)$.
\end{theorem}

\begin{proof}The Riemannian curvature tensor $\widetilde{R}$ is characterized by
\begin{eqnarray}\label{eq_G}
\widetilde{R}(X,Y)Z =\widetilde{\nabla}_{X}\widetilde{\nabla}_{Y}Z-\widetilde{\nabla}_{Y}\widetilde{\nabla}_{X}Z-\widetilde{\nabla}_{[X,Y]}Z,
\end{eqnarray}
for each vector fields $X,Y$ and $Z$ on $M$. Using \eqref{eq_D} and \eqref{eq_E}, we obtain:	 	
\begin{eqnarray}\label{eq_H}
\widetilde{\nabla}_{X}\widetilde{\nabla}_{Y}Z&=&\widetilde{\nabla}_{X}\Big(\nabla_{X}Y+\frac{X(\alpha)}{2\alpha}Y+\frac{Y(\alpha)}{2\alpha}X-\frac{1}{2\alpha^{2}}g^{\alpha}(X,Y)grad^{}\alpha\Big)\notag\\
&=& \nabla_{X}\nabla_{Y}Z+\dfrac{X(\alpha)}{2\alpha}\nabla_{Y}Z+\dfrac{Y(\alpha)}{2\alpha}\nabla_{X}Z+\dfrac{Z(\alpha)}{2\alpha}\nabla_{X}Y\notag\\
&&+\big(\dfrac{Y(\alpha)Z(\alpha)}{2\alpha^{2}}+\dfrac{(\nabla_{Y}Z)(\alpha)}{2\alpha}-\frac{\|grad^{}\alpha\|^{2}}{4\alpha^{3}}g^{\alpha}(Y,Z)\big)X\notag\\
&&+\big(\dfrac{XZ(\alpha)}{2\alpha}-\dfrac{X(\alpha)Z(\alpha)}{4\alpha^{2}}\big)Y+\big(\dfrac{XY(\alpha)}{2\alpha}-\dfrac{X(\alpha)Y(\alpha)}{4\alpha^{2}}\big)Z\notag\\
&&-\dfrac{1}{2\alpha}\big(g^{\alpha}(\nabla_{X}Y,Z)+g^{\alpha}(\nabla_{X}Z,Y)+(\nabla_{Y}Z,X)\big)grad^{}\alpha\notag\\
&&+\big(\dfrac{X(\alpha)}{2\alpha^{3}}g^{\alpha}(Y,Z)-\dfrac{Y(\alpha)}{4\alpha^{3}}g^{\alpha}(X,Z)-\dfrac{Z(\alpha)}{4\alpha^{3}}g^{\alpha}(X,Y)\big)grad^{}\alpha\notag\\
&&-\dfrac{1}{2\alpha^{2}}g^{\alpha}(Y,Z)\nabla_{X}grad^{}\alpha.
\end{eqnarray}
Simply, by replacing $X$ with $Y$ in $\widetilde{\nabla}_{X}\widetilde{\nabla}_{Y}Z$, we find: 
\begin{eqnarray}\label{eq_I}
\widetilde{\nabla}_{Y}\widetilde{\nabla}_{X}Z&=& \nabla_{Y}\nabla_{X}Z+\dfrac{Y(\alpha)}{2\alpha}\nabla_{X}Z+\dfrac{X(\alpha)}{2\alpha}\nabla_{Y}Z+\dfrac{Z(\alpha)}{2\alpha}\nabla_{Y}X\notag\\
&&+\big(\dfrac{X(\alpha)Z(\alpha)}{2\alpha^{2}}+\dfrac{(\nabla_{X}Z)(\alpha)}{2\alpha}-\frac{\|grad^{}\alpha\|^{2}}{4\alpha^{3}}g^{\alpha}(X,Z)\big)Y\notag\\
&&+\big(\dfrac{YZ(\alpha)}{2\alpha}-\dfrac{Y(\alpha)Z(\alpha)}{4\alpha^{2}}\big)X+\big(\dfrac{YX(\alpha)}{2\alpha}-\dfrac{Y(\alpha)X(\alpha)}{4\alpha^{2}}\big)Z\notag\\
&&-\dfrac{1}{2\alpha}\big(g^{\alpha}(\nabla_{Y}X,Z)+g^{\alpha}(\nabla_{Y}Z,X)+(\nabla_{X}Z,Y)\big)grad^{}\alpha\notag\\
&&+\big(\dfrac{Y(\alpha)}{2\alpha^{3}}g^{\alpha}(X,Z)-\dfrac{X(\alpha)}{4\alpha^{3}}g^{\alpha}(Y,Z)-\dfrac{Z(\alpha)}{4\alpha^{3}}g^{\alpha}(Y,X)\big)grad^{}\alpha\notag\\
&&-\dfrac{1}{2\alpha^{2}}g^{\alpha}(X,Z)\nabla_{Y}grad^{}\alpha.
\end{eqnarray}
We also find: 
\begin{eqnarray}\label{eq_J}
\widetilde{\nabla}_{[X,Y]}Z&=& \nabla_{[X,Y]}Z+\dfrac{[X,Y](\alpha)}{2\alpha}Z+\dfrac{Z(\alpha)}{2\alpha}[X,Y]\notag\\
&&-\dfrac{1}{2\alpha^{2}}g^{\alpha}([X,Y], Z)grad^{}\alpha.
\end{eqnarray}
Substituting \eqref{eq_H}, \eqref{eq_I} and \eqref{eq_J} into \eqref{eq_G} we find \eqref{eq_F}.
\end{proof}

\begin{corollary}\label{cor_02} 
Given a para-K\"{a}hler-Norden manifold $(M^{2m},F,g)$. If $f$ is Killing potential, then $\widetilde{R}$ is expressed by: 
\begin{eqnarray*}
\widetilde{R}(X,Y)Z &=&R(X,Y)Z+\left( \frac{3Y(\alpha)Z(\alpha)}{4\alpha^{2}}-\frac{\|grad^{}\alpha\|^{2}}{4\alpha^{3}}g^{\alpha}(Y,Z)\right) X\nonumber\\
&&-\left( \frac{3X(\alpha)Z(\alpha)}{4\alpha^{2}}-\frac{\|grad^{}\alpha\|^{2}}{4\alpha^{3}}g^{\alpha}(X,Z)\right) Y\nonumber\\
&&+\left(\frac{3X(\alpha)}{4\alpha^{3}}g^{\alpha}(Y,Z)-\frac{3Y(\alpha)}{4\alpha^{3}}g^{\alpha}(X,Z)\right)  grad^{}\alpha,
\end{eqnarray*}%
for each vector fields $X,Y$ and $Z$ on $M$.
\end{corollary}

Let $K(X, Y)$ $($ respectively, $\widetilde{K}(X, Y)$$)$ be the sectional curvature of the plane spanned by $\{X,Y\}$ of $(M^{2m},F,g)$ $($respectively, $(M^{2m}, g^{\alpha})$$)$, for each $X$ and $Y$ two vector fields orthonormal with respect to $g$.

\begin{theorem}\label{th_3}
Given a para-K\"{a}hler-Norden manifold $(M^{2m},F,g)$, then $\widetilde{K}(X, Y)$ is expressed by:
\begin{eqnarray}\label{eq_K}
\widetilde{K}(X,Y)&=&-\frac{\|grad^{}\alpha\|^{2}}{4\alpha^{3}}+\dfrac{1}{\alpha(1+g(X,FV)^{2}+g(Y,FV)^{2})}\Big[K(X,Y)\nonumber\\
&&+\left(\dfrac{3Y(\alpha)^{2}}{4\alpha^{2}}-\dfrac{Hess_{\alpha}(Y,Y)}{2\alpha}\right)\left(1+g(X, FV)^{2}\right)\nonumber\\
&&+\left(\dfrac{3X(\alpha)^{2}}{4\alpha^{2}}-\dfrac{Hess_{\alpha}(X,X)}{2\alpha}\right)\left(1+g(Y, FV)^{2}\right)\nonumber\\
&&-\left(\dfrac{3X(\alpha)Y(\alpha)}{2\alpha^{2}}-\dfrac{Hess_{\alpha}(X,Y)}{\alpha}\right)g(X, FV)g(Y, FV)\Big].	
\end{eqnarray}
\end{theorem}

\begin{proof} We have:
\begin{equation}\label{eq_L}
\widetilde{K}(X,Y)=\dfrac{g^{\alpha}(\widetilde{R}(X,Y)Y,X)}{g^{\alpha}(X,X)g^{\alpha}(Y,Y)-g^{\alpha}(X,Y)^{2}}.
\end{equation} 
On the one hand, we have: 		
\begin{eqnarray}\label{eq_M}
\qquad\qquad g^{\alpha}(\widetilde{R}(X,Y)Y,X)= \alpha g(\widetilde{R}(X,Y)Y, X)+\alpha g(X,FV)g(\widetilde{R}(X,Y)Y, FV).
\end{eqnarray}
Direct calculation using \eqref{eq_B} and \eqref{eq_F} yields the following results:
\begin{eqnarray}\label{eq_N}
\alpha g(\widetilde{R}(X,Y)Y,X)&=&\alpha g(R(X,Y)Y,X)-\dfrac{1}{2}Hess_{\alpha}(X,X)\left(1+ g(Y,FV)^{2}\right) \nonumber\\
&&+\dfrac{1}{2}Hess_{\alpha}(X,Y)g(X,FV)g(Y,FV)-\dfrac{1}{2}Hess_{\alpha}(Y,Y)\notag\\
&&-\dfrac{3X(\alpha)Y(\alpha)}{4\alpha}g(X,FV)g(Y,FV)+\dfrac{Y(\alpha)^{2}}{4\alpha} \nonumber\\
&&+\dfrac{3X(\alpha)^{2}}{4\alpha}\left( 1+g(Y,FV)^{2}\right)-\frac{\|grad^{}\alpha\|^{2}}{4\alpha}\left( 1+g(Y,FV)^{2}\right),
\end{eqnarray}
\begin{eqnarray}\label{eq_O}
\alpha g(X,FV)g(\widetilde{R}(X,Y)Y, FV)&=&\dfrac{1}{2}Hess_{\alpha}(X,Y)g(X,FV)g(Y,FV)\nonumber\\
&&-\dfrac{1}{2}Hess_{\alpha}(Y,Y)g(X,FV)^{2}+\dfrac{3Y(\alpha)^{2}}{4\alpha}g(X,FV)^{2}\notag\\
&&-\dfrac{3X(\alpha)Y(\alpha)}{4\alpha}g(X,FV)g(Y,FV)\nonumber\\
&&-\frac{\|grad^{}\alpha\|^{2}}{4\alpha}g(X,FV)^{2},
\end{eqnarray}
Substituting \eqref{eq_N} and \eqref{eq_O} into \eqref{eq_M}, we find: 
\begin{eqnarray}\label{eq_P}
g^{\alpha}(\widetilde{R}(X,Y)Y,X)&=&+\alpha g(R(X,Y)Y,X)\nonumber\\
&&-\frac{\|grad^{}\alpha\|^{2}}{4\alpha}(1+g(X,FV)^{2}+g(Y,FV)^{2})\nonumber\\
&&+\left(\dfrac{3Y(\alpha)^{2}}{4\alpha}-\dfrac{Hess_{\alpha}(Y,Y)}{2}\right)\left(1+g(X, FV)^{2}\right)\nonumber\\
&&+\left(\dfrac{3X(\alpha)^{2}}{4\alpha}-\dfrac{Hess_{\alpha}(X,X)}{2}\right)\left(1+g(Y, FV)^{2}\right)\nonumber\\
&&-\left(\dfrac{3X(\alpha)Y(\alpha)}{2\alpha}-Hess_{\alpha}(X,Y)\right)g(X, FV)g(Y, FV).	
\end{eqnarray}
However, we also have:
\begin{eqnarray}\label{eq_Q}
\qquad\qquad g^{\alpha}(X,X)g^{\alpha}(Y,Y)-g^{\alpha}(X,Y)^{2}=\alpha^{2}(1+g(X,FV)^{2}+g(Y,FV)^{2}).
\end{eqnarray}
Finally, substituting \eqref{eq_P} and \eqref{eq_Q} into \eqref{eq_L} we find \eqref{eq_K}.
\end{proof}

\begin{corollary}\label{cor_03}
Given a para-K\"{a}hler-Norden manifold $(M^{2m},F,g)$. If $f$ is Killing
potential, then $\widetilde{K}(X, Y)$ is expressed by:
\begin{eqnarray*}
\widetilde{K}(X,Y)&=&-\frac{\|grad^{}\alpha\|^{2}}{4\alpha^{3}}+\dfrac{1}{\alpha(1+g(X,FV)^{2}+g(Y,FV)^{2})}\Big[K(X,Y)\\
&&+\dfrac{3X(\alpha)^{2}}{4\alpha^{2}}\left(1+g(Y, FV)^{2}\right)+\dfrac{3Y(\alpha)^{2}}{4\alpha^{2}}\left(1+g(X, FV)^{2}\right)\\
&&-\dfrac{3X(\alpha)Y(\alpha)}{2\alpha^{2}}g(X, FV)g(Y, FV)\Big].	
\end{eqnarray*}
\end{corollary}

\begin{remark}\label{rem_1}
If we assume that $\{e_{i}\}_{i=1,\ldots,2m}$ is a local orthonormal frame on $(M^{2m},F,g)$, such that $e_{2m}=FV$, then orthonormal vector fields
\begin{equation}\label{eq_R}
\tilde{e}_{i}= \dfrac{1}{\sqrt{\alpha}}e_{i},\; i=1,\ldots,2m-1,\; \tilde{e}_{2m}=\dfrac{1}{\sqrt{2\alpha}}FV
\end{equation}
is a local orthonormal frame on $(M^{2m}, g^{\alpha})$.	
\end{remark}

\begin{theorem}\label{th_4}
Given a para-K\"{a}hler-Norden manifold $(M^{2m},F,g)$. If $Ricci$  $($ respectively, $\widetilde{Ricci}$$)$  represents the Ricci tensor of $(M^{2m},F,g)$ $($ respectively, $(M^{2m}, g^{\alpha})$$)$, then we obtain:
\begin{eqnarray}\label{eq_S}
\widetilde{Ricci}(X) &=&\dfrac{1}{\alpha}Ricci(X)-\dfrac{m-1}{\alpha^{2}}\nabla_{X}grad^{}\alpha+\dfrac{3(m-1)X(\alpha)}{2\alpha^{3}}grad^{}\alpha\notag\\
&&-\left( \dfrac{(m-2)\|grad^{}\alpha\|^{2}}{2\alpha^{3}}+\dfrac{\Delta(\alpha)}{2\alpha^{2}}\right)X,
\end{eqnarray}
for each vector field $X$ on $M$.
\end{theorem}

\begin{proof} Given a local orthonormal frame $\{\tilde{e}_{i}\}_{i=1,\ldots,2m}$ on $(M^{2m}, g^{\alpha})$  defined by \eqref{eq_R}. We have
\begin{eqnarray}\label{eq_T}
\widetilde{Ricci}(X)&=&\sum_{i=1}^{2m}\widetilde{R}(X,\tilde{e}_{i})\tilde{e}_{i}\notag\\
&=& \dfrac{1}{\alpha}\sum_{i=1}^{2m-1}\widetilde{R}(X, e_{i})e_{i}+\dfrac{1}{2\alpha}\widetilde{R}(X, FV)FV.
\end{eqnarray}
From \eqref{curva-pure}, \eqref{eq_A}, \eqref{eq_B} and \eqref{eq_F} with direct computation we get:
\begin{eqnarray}\label{eq_U}
\frac{1}{\alpha}\sum_{i=1}^{2m}\widetilde{R}(X,e_{i})e_{i} &=&-\frac{2m-3}{2\alpha^{2}}\nabla_{X}grad^{}\alpha-\frac{1}{2\alpha^{2}}g(X,FV)\nabla_{FV}grad^{}\alpha\notag\\
&&+\frac{(6m-9)X(\alpha)}{4\alpha^{3}} grad^{}\alpha-\frac{\|grad^{}\alpha\|^{2}}{4\alpha^{3}}g(X,FV)FV\notag\\ 
&&-\left(\frac{\Delta(\alpha)}{2\alpha^{2}}+\frac{(2m-5)\|grad^{}\alpha\|^{2}}{4\alpha^{3}}\right)X+\frac{1}{\alpha} Ricc(X),
\end{eqnarray}
and
\begin{eqnarray}\label{eq_V}
\dfrac{1}{2\alpha}\widetilde{R}(X, FV)FV&=&-\frac{1}{2\alpha^{2}}\nabla_{X}grad^{}\alpha+\frac{1}{2\alpha^{2}}g(X,FV)\nabla_{FV} grad^{}\alpha\notag\\ 
&&-\frac{\|grad^{}\alpha\|^{2}}{4\alpha^{3}}X+\frac{3X(\alpha)}{4\alpha^{3}}grad^{}\alpha+\frac{\|grad^{}\alpha\|^{2}}{4\alpha^{3}}g(X,FV)FV.
\end{eqnarray}
Substituting \eqref{eq_U} and \eqref{eq_V} into \eqref{eq_T}, we find \eqref{eq_S}.
\end{proof}

\begin{corollary}\label{cor_04}
Given a para-K\"{a}hler-Norden manifold $(M^{2m},F,g)$. If $f$ is Killing potential, then we have:
\begin{eqnarray*}
\widetilde{Ricci}(X)&=&\dfrac{1}{\alpha}Ricci(X)+\dfrac{3(m-1)X(\alpha)}{2\alpha^{3}}grad^{}\alpha-\dfrac{(m-2)\|grad^{}\alpha\|^{2}}{2\alpha^{3}}X,	
\end{eqnarray*}
for each vector field $X$ on $M$.
\end{corollary}

\begin{theorem}\label{th_5}
Given a para-K\"{a}hler-Norden manifold $(M^{2m},F,g)$. Suppose $Ric$ $($respectively, $\widetilde{Ric}$$)$ represents the Ricci curvature of $(M^{2m},F,g)$ $($respectively, $(M^{2m}, g^{\alpha})$$)$, then we obtain:
\begin{eqnarray}\label{eq_W}
\widetilde{Ric}(X,Y) &=&Ric(X,Y)-\dfrac{m-1}{\alpha}Hess_{\alpha}(X,Y)+\dfrac{3(m-1)}{2\alpha^{2}}X(\alpha)Y(\alpha)\notag\\
&&-\left( \dfrac{(m-2)\|grad^{}\alpha\|^{2}}{2\alpha^{3}}+\dfrac{\Delta(\alpha)}{2\alpha^{2}}\right)g^{\alpha}(X,Y),
\end{eqnarray}
for each vector fields $X$ and $Y$ on $M$.
\end{theorem}

\begin{proof} Given a local orthonormal frame $\{\tilde{e}_{i}\}_{i=1,\ldots,2m}$ on $(M^{2m}, g^{\alpha})$  defined by \eqref{eq_R}. We have
\begin{eqnarray}\label{eq_X}
\widetilde{Ric}(X,Y)&=&g^{\alpha}(\widetilde{Ricci}(X),Y)\notag\\
&=&\alpha g(\widetilde{Ricci}(X),Y)+\alpha g(Y,FV)g(\widetilde{Ricci}(X),FV).
\end{eqnarray}
Direct calculation using \eqref{eq_S} yields the following:
\begin{eqnarray}\label{eq_Y}
\alpha g(\widetilde{Ricci}(X),Y)&=&Ric(X,Y)-\dfrac{m-1}{\alpha}Hess_{\alpha}(X,Y)+\dfrac{3(m-1)}{2\alpha^{2}}X(\alpha)Y(\alpha)\notag\\
&&-\left( \dfrac{(m-2)\|grad^{}\alpha\|^{2}}{2\alpha^{2}}+\dfrac{\Delta(\alpha)}{2\alpha}\right)g(X,Y),
\end{eqnarray}
and
\begin{eqnarray}\label{eq_Z}
\alpha g(Y,FV)g(\widetilde{Ricci}(X), FV)=-\left( \dfrac{(m-2)\|grad^{}\alpha\|^{2}}{2\alpha^{2}}+\dfrac{\Delta(\alpha)}{2\alpha}\right)g(X,FV)g(Y,FV).\notag\\
\end{eqnarray}
Substituting \eqref{eq_Y} and \eqref{eq_Z} into \eqref{eq_X}, we find \eqref{eq_W}.
\end{proof}

\begin{corollary}\label{cor_05}
Given a para-K\"{a}hler-Norden manifold $(M^{2m},F,g)$. If $f$ is Killing potential, then we have:
\begin{eqnarray*}
\widetilde{Ric}(X,Y)=Ric(X,Y)+\dfrac{3(m-1)}{2\alpha^{2}}X(\alpha)Y(\alpha)-\dfrac{(m-2)\|grad^{}\alpha\|^{2}}{2\alpha^{3}}g^{\alpha}(X,Y),
\end{eqnarray*}
for each vector fields $X$ and $Y$ on $M$.
\end{corollary}

\begin{theorem}\label{th_6}
Given a para-K\"{a}hler-Norden manifold $(M^{2m},F,g)$. Suppose $\sigma$ $($respectively, $\widetilde{\sigma}$$)$ denote the scalar curvature of $(M^{2m},F,g)$ $($respectively, $(M^{2m}, g^{\alpha})$$)$, then we obtain:
\begin{eqnarray*}
\widetilde{\sigma}=\dfrac{1}{\alpha}\sigma-\dfrac{2m-1}{\alpha^{2}}\Delta(\alpha)-\dfrac{(2m-1)(m-3)}{2\alpha^{3}}\|grad^{}\alpha\|^{2}.
\end{eqnarray*}
\end{theorem}

\begin{proof} Given a local orthonormal frame $\{\tilde{e}_{i}\}_{i=1,\ldots,2m}$ on $(M^{2m}, g^{\alpha})$  defined by \eqref{eq_R}. We have:
\begin{eqnarray*}		
\widetilde{\sigma}&=&\sum_{i=1}^{2m}\widetilde{Ric}(\tilde{e}_{i},\tilde{e}_{i})\notag\\
&=&\dfrac{1}{\alpha}\sum_{i=1}^{2m-1}\widetilde{Ric}(e_{i},e_{i})+\dfrac{1}{2\alpha}\widetilde{Ric}(FV,FV).
\end{eqnarray*}
Direct calculation using \eqref{eq_W} yields the following:
\begin{eqnarray*}
\widetilde{\sigma}&=&\dfrac{1}{\alpha}\sum_{i=1}^{2m-1}\Big(Ric(e_{i},e_{i})-\dfrac{(m-1)Hess_{\alpha}(e_{i},e_{i})}{\alpha}+\dfrac{3(m-1)e_{i}(\alpha)^{2}}{2\alpha^{2}}\notag\\
&&-\big(\dfrac{(m-2)\|grad^{}\alpha\|^{2}}{2\alpha^{3}}+\dfrac{\Delta(\alpha)}{2\alpha^{2}}\big)g^{\alpha}(e_{i},e_{i})\Big)\\
&&+\dfrac{1}{2\alpha}\Big(Ric(FV,FV)-\dfrac{(m-1)Hess_{\alpha}(FV,FV)}{\alpha}+\dfrac{3(m-1)(FV(\alpha))^{2}}{2\alpha^{2}}\notag\\
&&-\big(\dfrac{(m-2)\|grad^{}\alpha\|^{2}}{2\alpha^{3}}+\dfrac{\Delta(\alpha)}{2\alpha^{2}}\big)g^{\alpha}(FV,FV)\Big)\\
&=&\dfrac{1}{\alpha}\sigma-\dfrac{(m-1)\Delta(\alpha)}{\alpha^{2}}+\dfrac{3(m-1)\|grad^{}\alpha\|^{2}}{2\alpha^{3}}-\dfrac{(2m-1)\Delta(\alpha)}{2\alpha^{2}}\\
&&-\dfrac{(2m-1)(m-2)\|grad^{}\alpha\|^{2}}{2\alpha^{3}}-\dfrac{(m-2)\|grad^{}\alpha\|^{2}}{2\alpha^{3}}-\dfrac{\Delta(\alpha)}{2\alpha^{2}}\\
&=&\dfrac{1}{\alpha}\sigma-\dfrac{2m-1}{\alpha^{2}}\Delta(\alpha)-\dfrac{(2m-1)(m-3)}{2\alpha^{3}}\|grad^{}\alpha\|^{2}.
\end{eqnarray*}
\end{proof}

\begin{corollary}\label{cor_06}
Given a para-K\"{a}hler-Norden manifold $(M^{2m},F,g)$. If $\alpha$ is Killing potential, then we have:
\begin{eqnarray*}
\widetilde{\sigma}=\dfrac{1}{\alpha}\sigma-\dfrac{(2m-1)(m-3)}{2\alpha^{3}}\|grad^{}\alpha\|^{2}.	
\end{eqnarray*}
\end{corollary}
Below, we study the harmonicity and biharmonicity concerning the Berger-type metric conformal deformation.
Given a smooth map $\phi:(M^{m},g) \rightarrow (N^{n},h)$ between Riemannian manifolds. 
The map $\phi$ is said to be harmonic if it is a critical point of the energy functional
\begin{equation*}
E(\phi)=\dfrac{1}{2}\int_{D}Tr_{g}|d\phi|^{2}v_{g},
\end{equation*}
for any compact domain $D$ of $M$, where $v_{g}$ is the volume element of $M$. Or equivalently the map $F$ is harmonic if and only if it satisfies the associated Euler-Lagrange equations to the energy functional $E(\phi)$ given by 
\begin{equation*}
\tau(\phi)=Tr_{g}\nabla d\phi=0,
\end{equation*}
where $\tau(\phi)$ is the tension field of $\phi$. For more details see \cite{E.L,E.S,Kon}.

The biharmonic maps, which are a direct generalization of harmonic maps, are described as critical points of a bienergy functional.
\begin{equation*}
E_{2}(\phi)=\frac{1}{2} \int_{D}|\tau(\phi)|^{2}v_{g}.
\end{equation*}
Or equivalently the map $\phi$ is biharmonic if and only if it satisfies the Euler-Lagrange equation associated with the bienergy is given by the vanishing of the bitension field, 
\begin{eqnarray*}
\tau_{2}(\phi)=\Delta^{\phi}\tau(\phi)-Tr_{g}R^{N}(\tau(\phi),d\phi)d\phi.
\end{eqnarray*}
Here $\Delta^{\phi}\tau(\phi):=-Tr_{g}(\nabla^{\phi}_{\ast}\nabla^{\phi}_{\ast}-\nabla^{\phi}_{\nabla_{\ast}\ast})\tau(\phi)$ denotes the rough Laplacian of $\tau(\phi)$ on the pull-back bundle $\phi^{-1}TN$ and $R^{N}$ denotes the curvature tensor of the target manifold $N$. It is clear that every harmonic map is biharmonic. Therefore, it's interesting to construct non-harmonic biharmonic maps, also known as proper biharmonic maps. They have been studied and published in many papers see \cite{B.K,Bal,B.F.S}. The study of harmonic and biharmonic maps has grown to be one of the most important areas of differential geometry research todaysee \cite{B.W,B.Ch1,C.L.W,D.Z2,D.Z4,D.L.Z,Z.D3,Z.G}.

\section{Harmonicity concerning the Berger-type metric conformal deformation}\label{Sec4}
\subsection{The harmonicity of the identity map}\;\\
Let $(M^{2m},F,g)$ be a para-K\"{a}hler-Norden manifold, $\alpha : M\rightarrow]0,+\infty[$ be a strictly positive smooth function and $g^{\alpha}$ the Berger-type metric conformal deformation of $g$.

\begin{theorem}\label{th_7} Given the identity map $I:(M^{2m},F,g)\rightarrow (M^{2m}, g^{\alpha})$. Then the tension field $\widetilde{\tau}(I)$ of $I$ is expressed by:
\begin{eqnarray*}
\widetilde{\tau}(I)=\dfrac{1-2m}{2\alpha}grad^{}\alpha.
\end{eqnarray*}
\end{theorem}

\begin{proof} Given a local orthonormal frame $\{e_{i}\}_{i=1,\ldots,2m}$ on $M$, with $e_{2m}=FV$. We calculate the tension field $\widetilde{\tau}(I)$ of the Identity map $I$.
\begin{eqnarray*}
\widetilde{\tau}(I)&=&\sum_{i=1}^{2m}\left( \widetilde{\nabla}^{I}_{e_{i}}dI(e_{i})-dI(\nabla_{e_{i}}e_{i})\right)  \\
&=&\sum_{i=1}^{2m}\left( \widetilde{\nabla}_{dI(e_{i})}dI(e_{i})-\nabla_{e_{i}}e_{i}\right)  \\
&=&\sum_{i=1}^{2m}\left( \widetilde{\nabla}_{e_{i}}e_{i}-\nabla_{e_{i}}e_{i}\right).
\end{eqnarray*}
Using Theorem \ref{th_1}, we find:
\begin{eqnarray*}
\widetilde{\tau}(I)&=&\sum_{i=1}^{2m}\left( \nabla_{e_{i}}e_{i}+\frac{e_{i}(\alpha)}{\alpha}e_{i}-\frac{1}{2\alpha^{2}}g^{\alpha}(e_{i},e_{i})grad^{}\alpha-\nabla_{e_{i}}e_{i}\right) \\
&=&\sum_{i=1}^{2m}\frac{e_{i}(\alpha)}{\alpha}e_{i}-\frac{1}{2\alpha^{2}}\left( \sum_{i=1}^{2m-1}g^{\alpha}(e_{i},e_{i})+g^{\alpha}(FV,FV)\right) grad^{}\alpha\\
&=&\dfrac{1-2m}{2\alpha}grad^{}\alpha.
\end{eqnarray*}	
\end{proof}

We note that the identity map $I:(M^{2m},F,g)\rightarrow (M^{2m}, g^{\alpha})$  is harmonic if and only if \; $\alpha$is constant.

\begin{theorem}\label{th_8} Given the identity map $I:(M^{2m}, g^{\alpha})\rightarrow (M^{2m},F,g)$. Then the tension field $\widetilde{\tau}(I)$ of $I$ is expressed by:
\begin{eqnarray*}
\widetilde{\tau}(I)=\dfrac{m-1}{\alpha^{2}}grad^{}\alpha.
\end{eqnarray*}
\end{theorem}

\begin{proof}
Given a local orthonormal frame $\{\tilde{e}_{i}\}_{i=1,\ldots,2m}$ on $(M^{2m}, g^{\alpha})$  defined by \eqref{eq_R}. Then we have:
\begin{eqnarray*}
\widetilde{\tau}(I)&=&\sum_{i=1}^{2m}\left( \nabla_{\tilde{e}_{i}}\tilde{e}_{i}-\widetilde{\nabla}_{\tilde{e}_{i}}\tilde{e}_{i}\right) .
\end{eqnarray*}
Using theorem \ref{th_1}, we get:
\begin{eqnarray*}
\widetilde{\tau}(I)&=&\sum_{i=1}^{2m}\left( \nabla_{\tilde{e}_{i}}\tilde{e}_{i}-\nabla_{\tilde{e}_{i}}\tilde{e}_{i}-\dfrac{1}{\alpha}\tilde{e}_{i}(\alpha)\tilde{e}_{i}+\frac{1}{2\alpha^{2}}g^{\alpha}(\tilde{e}_{i},\tilde{e}_{i})grad^{}\alpha\right) \\
&=&-\dfrac{1}{\alpha^{2}}\left(\sum_{i=1}^{2m-1} e_{i}(\alpha)e_{i}+\dfrac{1}{2}FV(\alpha)FV\right)+\frac{1}{2\alpha^{2}}\sum_{i=1}^{2m}g^{\alpha}(\tilde{e}_{i},\tilde{e}_{i})grad^{}\alpha \\
&=&-\dfrac{1}{\alpha^{2}}grad^{}\alpha+\dfrac{m}{\alpha^{2}}grad^{}\alpha\\
&=&\dfrac{m-1}{\alpha^{2}}grad^{}\alpha.
\end{eqnarray*}	
\end{proof}

We note that the identity map $I:(M^{2m}, g^{\alpha})\rightarrow (M^{2m},F,g)$  is harmonic if and only if
$\alpha$is constant or $\dim M=2$.

\subsection{Harmonicity of the map $\varphi:(M^{m},g)\rightarrow (N^{2n},h^{\alpha})$}\quad\\
Let $(N^{2n},F,h)$ be a para-K\"{a}hler-Norden manifold, $\alpha : N\rightarrow]0,+\infty[$ be a strictly positive smooth function and $h^{\alpha}$ the Berger-type metric conformal deformation of $h$.

\begin{theorem}\label{th_9} Given a smooth map $\varphi:(M^{m} ,g)\rightarrow (N^{2n},h^{\alpha})$. Then
the tension field $\widetilde{\tau}(\varphi)$ of the map $\varphi$ is expressed by:
\begin{eqnarray*}
\widetilde{\tau}(\varphi)=\tau(\varphi)+\dfrac{1}{\alpha} d\varphi\left(grad^{M}(f\circ\varphi)\right)-\dfrac{1}{2\alpha^{2}}Tr_{g}h^{\alpha}(d\varphi(\ast),d\varphi(\ast))grad^{N}f.
\end{eqnarray*}
where  $\tau(\varphi)$ is the tension field of $\varphi:(M^{m}, g)\rightarrow (N^{2n},F,h)$.
\end{theorem}

\begin{proof} Given a local orthonormal frame $\{e_{i}\}_{i=1,\ldots,m}$ on $(M^{m}, g)$. We calculate the tension field $\widetilde{\tau}(\varphi)$ of the map $\varphi$.
\begin{eqnarray*}
\widetilde{\tau}(\varphi)&=&\sum_{i=1}^{m}\left( \widetilde{\nabla}^{\varphi}_{e_{i}}d\varphi(e_{i})-d\varphi(\nabla_{e_{i}}e_{i})\right) \\
&=&\sum_{i=1}^{m}\left( \widetilde{\nabla}^{N}_{d\varphi(e_{i})}d\varphi(e_{i})-d\varphi(\nabla_{e_{i}}e_{i})\right) \\
&=& \sum_{i=1}^{m}\left( \nabla^{N}_{d\varphi(e_{i})}d\varphi(e_{i})+\dfrac{d\varphi(e_{i})(\alpha)}{\alpha}d\varphi(e_{i})\right.\\
&&\left.-\dfrac{1}{2\alpha^{2}}h^{\alpha}(d\varphi(e_{i}), d\varphi(e_{i}))grad^{}\alpha-d\varphi(\nabla_{e_{i}}e_{i})\right) \\
&=&\tau(\varphi)+\dfrac{1}{\alpha}d\varphi\left(grad^{M}(f\circ\varphi)\right)-\dfrac{1}{2\alpha^{2}}Tr_{g}h^{\alpha}(d\varphi(\ast),d\varphi(\ast))grad^{N}f.
\end{eqnarray*}
\end{proof}

\begin{theorem}\label{th_10} Given a smooth map $\varphi:(M^{m} ,g)\rightarrow (N^{2n},h^{\alpha})$. Then
the map $\varphi$ is harmonic if and only if
\begin{eqnarray*}
\tau(\varphi)=-\dfrac{1}{\alpha}d\varphi\left(grad^{M}(f\circ\varphi)\right)+\dfrac{1}{2\alpha^{2}}Tr_{g}h^{\alpha}(d\varphi(\ast),d\varphi(\ast))grad^{N}f.
\end{eqnarray*}
\end{theorem}

\subsection{Harmonicity of the map $\phi:(M^{2m},g^{\alpha})\rightarrow (N^{n},h)$}\quad\\
Let $(M^{2m},F,g)$ be a para-K\"{a}hler-Norden manifold and $\alpha : M\rightarrow]0,+\infty[$ be a strictly positive smooth function, where $g^{\alpha}$ the Berger-type metric conformal deformation of $g$.

\begin{theorem}\label{th_11} Given a smooth map $\phi:(M^{2m},g^{\alpha})\rightarrow (N^{n},h)$. Then
the tension field $\widetilde{\tau}(\phi)$ of the map $\phi$ is expressed by:
\begin{eqnarray*}
\widetilde{\tau}(\phi)&=&\dfrac{1}{\alpha}\tau(\phi)-\dfrac{1}{2\alpha}\nabla^{N}_{d\phi(FV)}d\phi(FV)+\dfrac{m-1}{\alpha^{2}}d\phi(grad^{}\alpha),
\end{eqnarray*}
where  $\tau(\phi)$ is the tension field of $\phi:(M^{2m},F, g)\rightarrow (N^{n},h)$.
\end{theorem}

\begin{proof} Given a local orthonormal frame $\{\tilde{e}_{i}\}_{i=1,\ldots,2m}$ on $(M^{2m}, g^{\alpha})$  defined by \eqref{eq_R}. Then we have:
\begin{eqnarray}\label{eq_AB}
\widetilde{\tau}(\phi)&=&\sum_{i=1}^{2m}\left( \nabla^{\phi}_{\tilde{e}_{i}}d\phi(\tilde{e}_{i})-d\phi(\widetilde{\nabla}_{\tilde{e}_{i}}\tilde{e}_{i})\right),
\end{eqnarray}
where  $\nabla^{\phi}$ is the pull-back connection induced by $\phi:(M^{2m},g^{\alpha})\rightarrow (N^{n},h)$.\\
By direct calculations, we obtain:
\begin{eqnarray}\label{eq_AC}
\sum_{i=1}^{2m}\nabla^{\phi}_{\tilde{e}_{i}}d\phi(\tilde{e}_{i})&=&\dfrac{1}{\sqrt{\alpha}}\sum_{i=1}^{2m-1}\nabla^{\phi}_{e_{i}}\left(\dfrac{1}{\sqrt{\alpha}}d\phi(e_{i})\right)+\dfrac{1}{\sqrt{2\alpha}}\nabla^{\phi}_{FV}\left(\dfrac{1}{\sqrt{2\alpha}}d\phi(FV)\right)\nonumber\\
&=&\dfrac{-1}{2\alpha^{2}}d\phi(grad^{}\alpha)+\dfrac{1}{\alpha}\sum_{i=1}^{2m-1}\nabla^{\phi}_{e_{i}}d\phi(e_{i})+\dfrac{1}{2\alpha}\nabla^{\phi}_{FV}d\phi(FV)\nonumber\\
&=&\dfrac{-1}{2\alpha^{2}}d\phi(grad^{}\alpha)-\dfrac{1}{2\alpha}\nabla^{\phi}_{FV}d\phi(FV)+\dfrac{1}{\alpha}\sum_{i=1}^{2m}\nabla^{\phi}_{e_{i}}d\phi(e_{i}).
\end{eqnarray}
By similar computations, we obtain:
\begin{eqnarray*}
\sum_{i=1}^{2m}d\phi\left(\widetilde{\nabla}_{\tilde{e}_{i}}\tilde{e}_{i}\right)&=&\sum_{i=1}^{2m-1}d\phi\left( \widetilde{\nabla}_{\tilde{e}_{i}}\tilde{e}_{i}\right) +d\phi\left(\widetilde{\nabla}_{\tilde{e}_{2m}}\tilde{e}_{2m}\right)\\
&=&\dfrac{-1}{2\alpha^{2}}d\phi(grad^{}\alpha)+\dfrac{1}{\alpha}\sum_{i=1}^{2m-1}d\phi\left(\widetilde{\nabla}_{e_{i}}e_{i}\right) +\dfrac{1}{2\alpha}d\phi\left(\widetilde{\nabla}_{FV}FV\right).
\end{eqnarray*}
Using Theorem \ref{th_1}, we find: 
\begin{eqnarray}\label{eq_AD}
\sum_{i=1}^{2m}d\phi\left(\widetilde{\nabla}_{\tilde{e}_{i}}\tilde{e}_{i}\right)
&=&\dfrac{-1}{2\alpha^{2}}d\phi(grad^{}\alpha)-\dfrac{1}{2\alpha}d\phi\left(\frac{1}{2\alpha^{2}}g^{\alpha}(FV,FV)grad^{}\alpha)\right)\nonumber\\
&&+\dfrac{1}{\alpha}\sum_{i=1}^{2m-1}d\phi\left(\nabla_{e_{i}}e_{i}+\frac{e_{i}(\alpha)}{\alpha}e_{i}-\frac{1}{2\alpha^{2}}g^{\alpha}(e_{i},e_{i})grad^{}\alpha\right)\nonumber\\
&=&\dfrac{1-2m}{2\alpha^{2}}d\phi(grad^{}\alpha)+\dfrac{1}{\alpha}\sum_{i=1}^{2m}d\phi\nabla_{e_{i}}e_{i}.
\end{eqnarray}
In fact, by adding \eqref{eq_AC} and \eqref{eq_AD} in \eqref{eq_AB}, the proof is completed.
\end{proof}

\begin{theorem}\label{th_12} Given a smooth map $\phi:(M^{2m},g^{\alpha})\rightarrow (N^{n},h)$. Then
the map $\phi$ is harmonic if and only if
\begin{eqnarray*}
\tau(\phi)=\dfrac{1}{2}\nabla^{N}_{d\phi(FV)}d\phi(FV)-\dfrac{m-1}{\alpha}d\phi(grad^{}\alpha).
\end{eqnarray*}
\end{theorem}

%
%
%

\section{The biharmonicity of identity map}\label{Sec5}
\subsection{The biharmonicity of identity map $I:(M^{2m},F,g)\rightarrow (M^{2m}, g^{\alpha})$}
\begin{theorem}\label{th_14}
The bitension field $\widetilde{\tau}_{2}(I)$ of the identity map\\ $I:(M^{2m},F,g)\rightarrow (M^{2m}, g^{\alpha})$ is given by:
\begin{eqnarray}\label{eq_AE}
\widetilde{\tau}_{2}(I)&=&\dfrac{1-2m}{2\alpha} \left[\Delta^{I}grad^{}\alpha-Ricci(grad^{}\alpha)+\dfrac{2m-1}{2\alpha}\nabla_{grad^{}\alpha}grad^{}\alpha\right. \nonumber\\
&&\left.-\left(\dfrac{(2m+3)\|grad^{}\alpha\|^{2}}{4\alpha^{2}}-\dfrac{2\Delta(\alpha)}{\alpha}\right) grad^{}\alpha-\nabla_{FV}\nabla_{FV}grad^{}\right],  
\end{eqnarray}
where $\Delta^{I}grad^{}\alpha=-Tr_{g}(\nabla_{\ast}\nabla_{\ast}-\nabla_{\nabla_{\ast}\ast})grad^{}\alpha$.
\end{theorem}

\begin{proof}
 Given a local orthonormal frame $\{e_{i}\}_{i=1,\ldots,2m}$ on $M$, with $e_{2m}=FV$.
The bitension field of the identity map $I:(M^{2m},F,g)\rightarrow (M^{2m}, g^{\alpha})$ is given by:
\begin{eqnarray}\label{eq_AF}
\widetilde{\tau}_{2}(I)&=&\widetilde{\Delta}^{I}\tau(I)-Tr_{g}\widetilde{R}(\tau(I),dI)dI.
\end{eqnarray} 
From Theorem $\ref{th_7}$, we have:
\begin{eqnarray*}
\tau(I)=\dfrac{1-2m}{2\alpha}grad^{}\alpha.
\end{eqnarray*}
First we calculate $\widetilde{\Delta}^{I}\tau(I)$, by Theorem $\ref{th_1}$, we get:
\begin{eqnarray}\label{eq_AG}
\widetilde{\Delta}^{I}\tau(I)&=&\sum_{i=1}^{2m}\big(\widetilde{\nabla}_{e_{i}}\widetilde{\nabla}_{e_{i}}(\dfrac{1-2m}{2\alpha}grad^{}\alpha)-\widetilde{\nabla}_{\nabla_{e_{i}}e_{i}}(\dfrac{1-2m}{2\alpha}grad^{}\alpha)\big)\nonumber\\
&=&\displaystyle\dfrac{1-2m}{2\alpha} \Big[\Delta^{I}grad^{}\alpha+\left(\dfrac{(2m-3)\|grad^{}\alpha\|^{2}}{4\alpha^{2}}+\dfrac{3\Delta(\alpha)}{2\alpha}\right) grad^{}\alpha \nonumber\\
&&-\nabla_{FV}\nabla_{FV}grad^{}\Big].  
\end{eqnarray}
We also calculate $-Tr_{g}\widetilde{R}(\widetilde{\tau}(I),dI)dI$, by Theorem $\ref{th_2}$, we find:
\begin{eqnarray}\label{eq_AH}
-Tr_{g}\widetilde{R}(\widetilde{\tau}(I),dI)dI&=&\dfrac{2m-1}{2\alpha}\sum_{i=1}^{2m}\widetilde{R}(grad^{}\alpha,dI(e_{i}))dI(e_{i})\nonumber\\
&=&\dfrac{1-2m}{2\alpha} \left[-Ricci(grad^{}\alpha)+\dfrac{2m-1}{2\alpha}\nabla_{grad^{}\alpha}grad^{}\alpha\right. \nonumber\\
&&\left.-\left(\dfrac{m\|grad^{}\alpha\|^{2}}{\alpha^{2}}-\dfrac{\Delta(\alpha)}{2\alpha}\right)grad^{}\alpha\right] . 
\end{eqnarray}
Substituting $(\ref{eq_AG})$ and $(\ref{eq_AH})$ into $(\ref{eq_AF})$, we find $(\ref{eq_AE})$.
\end{proof}

\begin{theorem}\label{th_15}
The identity map $I:(M^{2m},F,g)\rightarrow (M^{2m}, g^{\alpha})$ is a biharmonic if
and only if 
\begin{eqnarray*}
&&\Delta^{I}grad^{}\alpha-Ricci(grad^{}\alpha)+\dfrac{2m-1}{2\alpha}\nabla_{grad^{}\alpha}grad^{}\alpha\\
&&-\left(\dfrac{(2m+3)\|grad^{}\alpha\|^{2}}{4\alpha^{2}}-\dfrac{2\Delta(\alpha)}{\alpha}\right) grad^{}\alpha-\nabla_{FV}\nabla_{FV}grad^{}=0.
\end{eqnarray*}
\end{theorem}

\begin{theorem}\label{th_16}
The identity map $I:(M^{2m},F,g)\rightarrow (M^{2m}, g^{\alpha})$ is a proper biharmonic if
and only if the function $\alpha$ is a non constant on $M$, and 
\begin{eqnarray*}
&&\Delta^{I}grad^{}\alpha-Ricci(grad^{}\alpha)+\dfrac{2m-1}{2\alpha}\nabla_{grad^{}\alpha}grad^{}\alpha \nonumber\\
&&-\left(\dfrac{(2m+3)\|grad^{}\alpha\|^{2}}{4\alpha^{2}}-\dfrac{2\Delta(\alpha)}{\alpha}\right) grad^{}\alpha-\nabla_{FV}\nabla_{FV}grad^{}=0.
\end{eqnarray*}
\end{theorem}

\subsection{The biharmonicity of identity map $I:(M^{2m}, g^{\alpha})\rightarrow (M^{2m},F,g)$}
\begin{theorem}\label{th_17}
The bitension field $\widetilde{\tau}_{2}(I)$ of the identity map \\$I:(M^{2m}, g^{\alpha})\rightarrow (M^{2m},F,g)$ is given by:
\begin{eqnarray}\label{eq_AI}
\widetilde{\tau}_{2}(I)
&=&\dfrac{m-1}{\alpha^{3}}\left[ \Delta^{I}grad^{}\alpha-Ricci(grad^{}\alpha)-\dfrac{m-5}{\alpha}\nabla_{grad^{}\alpha}grad^{}\alpha\right. \nonumber\\
&&\left. +2\left(\dfrac{(m-4)\|grad^{}\alpha\|^{2}}{\alpha^{2}}+\dfrac{\Delta(\alpha)}{\alpha}\right)grad^{}\alpha-\dfrac{1}{2}\nabla_{FV}\nabla_{FV}grad^{}\alpha\right].
\end{eqnarray}
\end{theorem}

\begin{proof}
 Given a local orthonormal frame $\{\tilde{e}_{i}\}_{i=1,\ldots,2m}$ on $(M^{2m}, g^{\alpha})$  defined by \eqref{eq_R}. The bitension field $\widetilde{\tau}_{2}(I)$ of the identity map\\ $I:(M^{2m},g^{\alpha})\rightarrow(M^{2m},F,g) $ is given by:
\begin{eqnarray}\label{eq_AJ}
\widetilde{\tau}_{2}(I)&=&\Delta^{I}\widetilde{\tau}(I)-Tr_{g^{\alpha}}R(\widetilde{\tau}(I),dI)dI.
\end{eqnarray}
From Theorem $\ref{th_8}$, we have:
\begin{eqnarray*}
\widetilde{\tau}(I)=\dfrac{m-1}{\alpha^{2}}grad^{}\alpha.
\end{eqnarray*}
We use the Theorem $\ref{th_1}$, to calculate $\Delta^{I}\widetilde{\tau}(I)$, then we find:
\begin{eqnarray}\label{eq_AK}
\Delta^{I}\widetilde{\tau}(I)&=&-\sum_{i=1}^{2m}\left( \nabla_{\tilde{e}_{i}}\nabla_{\tilde{e}_{i}}(\dfrac{m-1}{\alpha^{2}}grad^{}\alpha)-\nabla_{\widetilde{\nabla}_{\tilde{e}_{i}}\tilde{e}_{i}}(\dfrac{m-1}{\alpha^{2}}grad^{}\alpha)\right) \nonumber\\
&=&\dfrac{m-1}{\alpha^{3}}\left[ \Delta^{I}grad^{}\alpha+2\left(\dfrac{(m-4)\|grad^{}\alpha\|^{2}}{\alpha^{2}}+\dfrac{\Delta(\alpha)}{\alpha}\right)grad^{}\alpha\right. \nonumber\\
&&\left. -\dfrac{m-5}{\alpha}\nabla_{grad^{}\alpha}grad^{}\alpha-\dfrac{1}{2}\nabla_{FV}\nabla_{FV}grad^{}\alpha\right].
\end{eqnarray}
We also find
\begin{eqnarray}\label{eq_AL}
-Tr_{g^{\alpha}}R(\widetilde{\tau}(I),dI)dI&=&-\dfrac{m-1}{\alpha^{2}}\sum_{i=1}^{2m}R(grad^{}\alpha,\tilde{e}_{i})\tilde{e}_{i}\nonumber\\
&=&-\dfrac{m-1}{\alpha^{3}}Ricci(grad^{}\alpha). 
\end{eqnarray}
Substituting $(\ref{eq_AK})$ and $(\ref{eq_AL})$ into $(\ref{eq_AJ})$, we find $(\ref{eq_AI})$.
\end{proof}

\begin{theorem}\label{th_18}
The identity map $I:(M^{2m}, g^{\alpha})\rightarrow (M^{2m},F,g)$ is a biharmonic if
and only if $m=1$ or
\begin{eqnarray*}
&& \Delta^{I}grad^{}\alpha-Ricci(grad^{}\alpha)-\dfrac{m-5}{\alpha}\nabla_{grad^{}\alpha}grad^{}\alpha \nonumber\\
&& +2\left(\dfrac{\Delta(\alpha)}{\alpha}+\dfrac{(m-4)\|grad^{}\alpha\|^{2}}{\alpha^{2}}\right)grad^{}\alpha-\dfrac{1}{2}\nabla_{FV}\nabla_{FV}grad^{}\alpha=0.
\end{eqnarray*}
\end{theorem}

\begin{theorem}\label{th_19}
The identity map $I:(M^{2m}, g^{\alpha})\rightarrow (M^{2m},F,g)$ is a proper biharmonic if
and only if $\alpha\neq$ const, $m\neq1$ and
\begin{eqnarray*}
&& \Delta^{I}grad^{}\alpha-Ricci(grad^{}\alpha)-\dfrac{m-5}{\alpha}\nabla_{grad^{}\alpha}grad^{}\alpha \nonumber\\
&&+2\left(\dfrac{\Delta(\alpha)}{\alpha}+\dfrac{(m-4)\|grad^{}\alpha\|^{2}}{\alpha^{2}}\right)grad^{}\alpha-\dfrac{1}{2}\nabla_{FV}\nabla_{FV}grad^{}\alpha=0.
\end{eqnarray*}
\end{theorem}


\end{document}